\documentclass[12pt]{article}
\usepackage{amsfonts}
\usepackage{stmaryrd}
\usepackage{graphicx}
\usepackage{epsfig}
\usepackage{amsmath}
\usepackage{amsthm}
\usepackage{amssymb}
\usepackage{mathrsfs}
\usepackage{float}
\usepackage{multirow}
\usepackage{verbatim}
\usepackage{fancyhdr}
\usepackage{subfigure}
\usepackage{color}
\usepackage{mathtools}
\usepackage{mathrsfs}
\usepackage{sectsty}
\usepackage[title]{appendix}
\usepackage{threeparttable}
\usepackage{dcolumn}
\usepackage{booktabs}
\usepackage{indentfirst}
\usepackage{setspace}
\usepackage{bm}
\usepackage{enumerate}

\usepackage[labelfont=bf,labelsep=period]{caption}
\newcolumntype{d}[1]{D{.}{.}{#1}}

\newtheorem{Thm}{Theorem}[section]
\newtheorem{Ex}{Example}[section]

\newtheorem{Lem}{Lemma}[section]

\newtheorem{Pro}{Proposition}[section]
\newtheorem{Remark}{Remark}[section]

\newcommand{\jr}{j+\frac{1}{2}}
\newcommand{\jl}{j-\frac{1}{2}}
\newcommand{\ir}{i+\frac{1}{2}}
\newcommand{\il}{i-\frac{1}{2}}
\newcommand{\rev}{}

\allowdisplaybreaks
\begin{document}
\pagenumbering{arabic}
\baselineskip=2pc

\vspace*{0.5in}

\begin{center}

{{\bf Sub-optimal convergence of discontinuous Galerkin methods with
central fluxes {\rev{for linear hyperbolic equations}} with
even degree polynomial approximations}}

\end{center}

\vspace{.03in}

\centerline{Yong Liu\footnote{School of Mathematical Sciences,
 University of Science and Technology of China,
Hefei, Anhui 230026, P.R. China.  E-mail: yong123@mail.ustc.edu.cn. Research supported by the China Scholarship Council.}
, Chi-Wang Shu\footnote{Division of
Applied Mathematics, Brown University, Providence, RI 02912, USA.
E-mail: chi-wang\_shu@brown.edu.  Research supported by NSF grant DMS-1719410.}
and Mengping Zhang\footnote{School of Mathematical Sciences,
University of Science and Technology of China,
Hefei, Anhui 230026, P.R. China, E-mail: mpzhang@ustc.edu.cn.
Research supported by NSFC grant 11871448.}}

\vspace{.1in}

\centerline{\bf Abstract}

\smallskip
In this paper, we theoretically and numerically verify that the discontinuous Galerkin
(DG) methods with central fluxes {\rev{for linear hyperbolic equations}} on non-uniform meshes have sub-optimal convergence
properties when measured in the $L^2$-norm for even degree polynomial approximations.
On uniform meshes, the optimal error estimates are provided for arbitrary number of
cells in one and multi-dimensions, improving previous results. The theoretical findings
are found to be sharp and consistent with numerical results.

\vfill


\pagenumbering{arabic}

\section{Introduction}
\label{sec1}
\setcounter{equation}{0}
\setcounter{figure}{0}
\setcounter{table}{0}

A fundamental form of energy transmission is wave propagation, which arises in many fields
of science, engineering and industry, such as petroleum engineering, geoscience,
telecommunication, and the defense industry (see \cite{Durran1999,KAMPANIS2008}). It is
important for these applications to study efficient and accurate numerical methods to
solve wave propagation problems. Experience reveals that energy-conserving numerical
methods, which conserve the discrete approximation of energy, are favorable, because they
are able to maintain the phase and shape of the waves more accurately, especially for
long-time simulation.

Various numerical approximations of wave problems modeled by linear hyperbolic systems
can be found in the literature. Here, we will focus on the classical Runge-Kutta DG
method of Cockburn and Shu \cite{RKDG2001}. There are several approaches to obtain an
optimal, energy conserving DG method.  Chung and Engquist \cite{Chung2009} presented an
optimal, energy conserving DG method for the acoustic wave equation on staggered grids.
Chou et al. \cite{ChouXing2014} proposed an optimal energy conserving DG using
alternating fluxes for the second order wave equation.  More recently, Fu and Shu
\cite{fushu2018} developed an optimal energy conserving DG method by introducing an
auxiliary zero function.

As is well known, the simplest energy conserving DG method for hyperbolic equations is
the one using central fluxes.  However, it has sub-optimal convergence of order $k$
measured in the $L^2$-norm when piece-wise polynomials of an {\em odd} degree $k$ are used;
see, e.g. \cite{MSW}.  When $k$ is even, we usually observe higher convergence rates than
$k$th order for a general regular non-uniform meshes, such as random perturbation over an
uniform mesh, see section \ref{Num}. In fact, many papers have mentioned that the optimal
convergence rates can be observed when even degree polynomials are used; see for
example \cite{YXINGKDV,Ycheng2017,MSW,JieDu2019Stability}. In this paper, we provide a counter example to
show that the scheme only has sub-optimal error accuracy of order $k$ for a regular
non-uniform mesh, when $k$ is even.  We refer to the work of Guzm\'{a}n and Rivi\`{e}re
\cite{JGuzmanJSC2009} in which they constructed a special mesh sequence to produce the
sup-optimal accuracy for the non-symmetric DG methods for elliptic problems when $k$ is
odd.  For uniform meshes, the classical DG scheme with the central flux does have
the optimal convergence rate $k+1$, observed in the numerical experiments and proved
theoretically under the condition that the number of cells in the mesh is $odd$
\cite{YXINGKDV,MSW}. In this paper, we provide a new proof which is available for
arbitrary number of cells and dimensions {\rev{for linear hyperbolic equations}}. We have used the shifting technique
\cite{liu2018cdg,liu2019pk} to construct the special local projection to obtain the
optimal error estimate on uniform meshes.  We also numerically find the superconvergence
phenomenon for the cell averages and numerical fluxes.

The outline of the paper is as follows. In section \ref{sec2}, we review the DG scheme for
hyperbolic equations with central fluxes and give the error estimates for the
semi-discrete version in one dimension. We extend our analysis to multi-dimensions
in section \ref{twodimension}. In section \ref{Num}, we give numerical examples to show
the sub-optimal convergence for non-uniform meshes and optimal convergence for uniform
meshes in both one and two-dimensional cases. Finally, we give concluding remarks in
section \ref{concluding}.  Some of the technical proof of the lemmas and propositions
is included in the Appendix \ref{appendix}.

\section{One dimensional problems}
\label{sec2}
\setcounter{equation}{0}
\setcounter{figure}{0}
\setcounter{table}{0}

We consider the following one dimensional linear hyperbolic equation
\begin{equation}
\label{1dlinear}
\left\{
\begin{aligned}
&u_{t}+u_{x}=0, \ x \in [0,1], \ t\geq 0\\
&u(x,0)=u_{0}(x), \ x \in [0,1],
\end{aligned}
\right.
\end{equation}
with periodic boundary condition. We first introduce the usual notations of the DG method.
For a given interval $\Omega=[0,1]$ and the index set $\mathbb{Z}_N=\{1,2,\ldots,N\}$, the usual DG mesh $\mathcal{I}_N$ is defined as:
\begin{align}
0=x_{\frac{1}{2}}<x_{\frac{3}{2}}<\ldots < x_{N+\frac{1}{2}}=1.
\end{align}
We denote
\begin{align}
I_{j}=(x_{j-\frac{1}{2}},x_{j+\frac{1}{2}}), \quad x_{j}=\frac{1}{2}(x_{j-\frac{1}{2}}+x_{j+\frac{1}{2}}),\quad h_j=x_{j+\frac{1}{2}}-x_{j-\frac{1}{2}}, \quad j \in \mathbb{Z}_N.
\end{align}
We also assume the mesh is regular, i.e., the ratio between the maximum and minimum mesh sizes shall stay bounded during mesh refinements. That means there exists a positive constant $\sigma\geq 1$, such that,
\begin{align}\label{meshreg}
\frac{1}{\sigma}h\leq h_j \leq \sigma h, \,\, h = \frac{1}{N}, \quad  \forall j \in \mathbb{Z}_N.
\end{align}
 We define the approximation space as
\begin{align}
V_h^k=\{v_h: (v_h)|_{I_j} \in \mathbb{P}^{k}(I_j),j=1,\ldots,N\}.
\end{align}
Here $\mathbb{P}^{k}(I_j)$ denotes the set of all polynomials of degree at most $k$ on $I_j$. We first introduce some standard Sobolev space notations. For any integer $m>0$, $W^{m,p}(D)$ denote the standard Sobolev spaces on the sub-domain $D\subset\Omega$ equipped with the norm $\|\cdot\|_{m,p,D}$ and the semi-norm $|\cdot|_{m,p,D}$. If $p=2$, we set $W^{m,p}(D)=H^m(D)$, and $|\cdot|_{m,p,D}=|\cdot|_{m,D}$ and we omit the index $D$, when $D=\Omega$.

The semi-discrete DG scheme is to seek $u_h\in V_h$ such that for all $v_h \in V_h$,
\begin{align}\label{DGscheme}
((u_h)_t,v_h)_j+a_j(u_h,v_h)=0, \quad \forall j \in \mathbb{Z}_N,
\end{align}
where
\begin{align}
\label{DGoperator}
a_{j}(u_h,v_h)=-(u_h,(v_h)_x)_j + \hat{u_h}v_h^{-}|_{\jr}-\hat{u_h}v_h^+|_{\jl},
\end{align}
where $(u,v)_j=\int_{I_j}u v \, dx$, $v^-|_{\jr}$ and $v^+|_{\jr}$
denote the left and right limits of $v$ at the point $x_{\jr}$,
respectively, and $\hat{u_h}$ are the numerical fluxes. Here, we consider the central flux,
\begin{align}\label{centralflx}
\hat{u_h}=\{u_h\}=\frac{1}{2}(u_h^-+u_h^+).
\end{align}
{\rev{
For the central flux, we have,
\begin{align}
\label{energy-conserving}
\sum_{j=1}^N a_j(u_h,u_h)=0, \quad \forall u_h \in V_h.
\end{align}
}}
The initial datum $u_h(x,0)=Pu_0$ is obtained by the standard $L^2$ projection,
\begin{align}
(u_0-Pu_0,v_h)_j=0, \quad \forall v_h \in \mathbb{P}^k(I_j).
\end{align}
Thus, we have,
\begin{align}\label{initerr}
\|u_0-u_h(\cdot,0)\| \lesssim h^{k+1}\|u\|_{k+1}.
\end{align}
Here and below, an unmarked norm $\|\cdot\|$ denotes the $L^2$ norm, and $A\lesssim B$ denotes that $A$ can be bounded by $B$ multiplied by a constant independent of the mesh size $h$.
As mentioned earlier, we have the following energy-conserving results \cite{MSW}.
\begin{Thm}\label{stab}
Suppose $u_h$ is the solution of DG scheme (\ref{DGscheme}), then it satisfies
\begin{align}
\frac{d}{dt}\|u_h\|^2=0.
\end{align}
\end{Thm}
Next we consider the error estimate, first we recall the following basic facts \cite{ciarlet2002finite}. For any function $w_h\in V_h$,
\begin{align}\label{inverseineq}
&(i) \ \|(w_h)_x\| \lesssim h^{-1} \|w_h\|,\nonumber\\
&(ii)\  \|w_h\|_{\Gamma_h}\lesssim h^{-\frac{1}{2}} \|w_h\| ,
\end{align}
where $\Gamma_h$ denotes the set of boundary points of all elements $I_j$, and
the norm $\|w_h\|_{\Gamma_h}=\left(\sum_{j=1}^N((w_h)_{\jr}^+)^2+((w_h)_{\jl}^-)^2\right)^{\frac{1}{2}}$.
In order to obtain the optimal error estimate for the case of uniform meshes,
we need to use the shifting technique \cite{liu2018cdg,liu2019pk} to construct a
special projection $P_h^{\star}$, which is defined as follows.
For any given function $w\in L^{\infty}(\Omega)$ and each $j$,
\begin{align}\label{shiftprojection}
\int_{I_j}P_{h}^\star w(x) \,dx =\int_{I_j} w(x) \, dx,\\
\widetilde{P_h}(P_{h}^\star w; v)_j=\widetilde{P_h}( w; v)_j \quad \forall v\in \mathbb{P}^{k}(I_j),
\end{align}
where $\widetilde{P_h}( w; v)_j$ is defined as
\begin{align}
\widetilde{P_h}( w; v)_j=-(w,v_x)_j+\frac{w(x_{\jr}^-)+w(x_{\jl}^+)}{2}(v(x_{\jr}^-)-v(x_{\jl}^+)).
\end{align}
Note that the projection $P_h^\star$ is a local projection, so we only consider the projection defined on the
reference interval $[-1,1]$. We have the following lemma to establish the fact that the projection is well defined.
\begin{Lem}\label{lemprojection}
When $k$ is even, the projection $P_h^\star$ defined by (\ref{shiftprojection}) on the interval $[-1,1]$ exists and is unique for any $L^\infty$ function $w$, and the projection is bounded in the $L^{\infty}$ norm, i.e.,
\begin{align}\label{projectionbound}
\|P_h^\star w\|_{\infty} \leq C(k) \|w\|_{\infty},
\end{align}
where $C(k)$ is a constant that only depends on $k$  but is independent of $w$.
\end{Lem}
\proof We provide the proof of this lemma in the appendix; see section \ref{proofoflemmaprojection}.
\begin{Remark}
The projection $P_h^\star$ is only well defined when $k$ is even. In fact, when $k$ is odd, for example $k=1$, we can take $w_I=x\in \mathbb{P}^{1}([-1,1])$, which satisfies
\begin{align}
&\int_{-1}^1 w_I(x) \,dx=0,\\
&\widetilde{P_h}(w_I; v)=-\int_{-1}^1 w_I(x) v_x \, dx+\frac{w_I(1)+w_I(-1)}{2}(v(1)-v(-1))=0, \quad \forall v \in \mathbb{P}^1([-1,1]).
\end{align}
It means that there exists a nonzero function $w_I=P_h^\star w$, where $w\equiv0$. This implies
that $P_h^\star w$ is not unique.
\end{Remark}

\begin{Remark}
In fact, the projection $P_h^\star$ has an equivalent definition as follows,
\begin{align}
\int_{x_{\jl}}^{x_{\jr}}P_h^\star w v \, dx = \int_{x_{\jl}}^{x_{\jr}} w v \, dx, \quad \forall v \in \mathbb{P}^{k-1}(I_j),\\
\frac{1}{2}(P_h^\star w(x_{\jr}^-)+P_h^\star w(x_{\jl}^+))=\frac{1}{2}(w(x_{\jr})+w(x_{\jl})).
\end{align}
\end{Remark}
As a direct corollary of lemma \ref{lemprojection} {\rev{and the locality of the projection}}, the standard approximation theory \cite{ciarlet2002finite} implies, for a smooth function $w$,
\begin{align}\label{proest}
\|P_h^\star w(x)- w(x)\| +h^{\frac{1}{2}}\|P_h^\star w(x)- w(x)\|_{\Gamma_h} \lesssim h^{k+1}\|w\|_{k+1}.
\end{align}
We also have the following properties of the projection $P_h^\star$,
\begin{Lem}
Suppose that $u=x^{k+1}$. Let $u_{j}=P_h^\star u|_{I_j}$.  If $h_{j-1}=h_{j}=h_{j+1}=h$, then we have the following relationship:
\begin{align}
(x-h)^{k+1}-u_{j-1}(x-h)=x^{k+1}- u_j(x)=(x+h)^{k+1}-u_{j+1}(x+h), \quad \forall x \in I_j.
\end{align}
where $P_h^\star u|_{I_j}$ means that the projection of $u$ is defined on the subinterval $I_j$, and $u_{j-1}(x-h)$, $u_{j+1}(x+h)$ refer to the projection of $u$ on the element $I_{j-1}$ and $I_{j+1}$ respectively, since $x\in I_j$ implies $(x-h)\in I_{j-1}$ and $(x+h)\in I_{j+1}$.
\end{Lem}
\proof The proof of this lemma is by the same arguments as
in \cite{liu2018cdg,liu2019pk}, so we omit it here.

By this lemma, we also have the following superconvergence results.
\begin{Pro}\label{superpro}
Given the index $j$, suppose that $u$ is a $(k+1)$th degree polynomial function in
$\mathbb{P}^{k+1}(I_{j-1}\bigcup I_j \bigcup I_{j+1})$.  If $h_{j-1}=h_{j}=h_{j+1}=h$, we have
\begin{align}
a_{j}(P_h^\star u,v_h)=a_j(u,v_h) \quad \forall v_h \in \mathbb{P}^{k}(I_j),
\end{align}
where $a_{j}$ is defined by (\ref{DGoperator}).
\end{Pro}
Then we can state the main theorem of this paper.
\begin{Thm}
Suppose $u_h$ is the numerical solution of the DG scheme (\ref{DGscheme}) for equation (\ref{1dlinear}) with a smooth initial condition $u(\cdot,0)\in H^{k+2}(\Omega)$, and $u$ is the exact solution of (\ref{1dlinear}), then the approximation $u_h$ satisfies the following $L^2$ error estimate:
\begin{align}\label{nonuniformest}
\|u(\cdot,T)-u_h(\cdot,T)\| \lesssim h^{k},
\end{align}
where $k$ is the degree of the piecewise polynomials in the finite element spaces $V_h$.
Furthermore, when $k$ is even and the mesh is uniform, we have the optimal error estimate:
\begin{align}\label{uniformest}
\|u(\cdot,T)-u_h(\cdot,T)\| \lesssim h^{k+1}.
\end{align}
\end{Thm}
\begin{proof}
Obviously, the exact solution $u$ of (\ref{1dlinear}) also satisfies
\begin{align}\label{exactuscheme}
(u_t,v_h)+a_{j}(u,v_h)=0, \quad \forall v_h\in V_h.
\end{align}
Subtracting (\ref{DGscheme}) from (\ref{exactuscheme}), we obtain the error equation
\begin{align}\label{errorequ}
((u-u_h)_t,v_h)_j+a_j(u-u_h,v_h)=0, \quad \forall v_h \in V_h.
\end{align}
We denote
\begin{align}
\xi=u_h-P^{\star} u; \quad \eta = u- P^{\star}u,
\end{align}
where $P^\star$ is some projection. From the error equation (\ref{errorequ}), and taking $v_h=\xi$, we have
\begin{align}\label{xietaestimate}
(\xi_t,\xi)_j+a_j(\xi,\xi)=(\eta_t,\xi)_j+a_j(\eta,\xi).
\end{align}
For the nonuniform mesh case, the sub-optimal error estimate can be easily obtained by using the
standard $L^2$ projection $P$.  We take $P^\star$ as the standard $L^2$ projection $P$, then we have,
\begin{align}\label{l2error}
\|u-P u\|+h^{\frac{1}{2}}\|u-P u\|_{\Gamma_h} \lesssim h^{k+1}\|u\|_{k+1}.
\end{align}
For the left-hand side of (\ref{xietaestimate}), we can use {\rev{(\ref{energy-conserving})}} to obtain
\begin{align}
\frac{1}{2}\frac{d}{dt}\|\xi\|^2&=-\sum_{j=1}^{N}\{\eta\}_{\jr}[\xi]_{\jr}\nonumber\\
&\lesssim h^{k}\|\xi\|\|u\|_{k+1},
\end{align}
where the last inequality is from (\ref{l2error}) and (ii) of (\ref{inverseineq}). Thus, by using Gronwall's inequality and (\ref{initerr}), we have,
\begin{align}
\|\xi\|\lesssim h^{k}\|u\|_{k+1}.
\end{align}
The triangle inequality implies our designed results for the general non-uniform mesh case.

For the case of uniform meshes, when $k$ is even, we take $P^{\star}$ as $P_h^\star$ which is
defined in (\ref{shiftprojection}).  Let $u_I^j$ be the Taylor expansion polynomial of order
$k+1$ of $u$ over $D_j=(x_{j-\frac{3}{2}},x_{j+\frac{3}{2}})$, i.e., $u_I^j(x)=\sum_{i=0}^{k+1}u^{(i)}(x_j)(x-x_j)^i$, $x\in D_j$. Let $r_u^j$ denote the remainder term, i.e., $r_u^j=u-u_I^j$. Recalling the Bramble-Hilbert lemma in \cite{ciarlet2002finite}, we have
\begin{align}\label{estimateremainder}
\|r_u^j\|_{L^\infty(D_j)} \lesssim h^{k+\frac{3}{2}}|u|_{k+2,D_j}.
\end{align}
Thus, using Proposition \ref{superpro}, we have
\begin{align*}
a_{j}(\eta,\xi)&=a_j(u_I^j-P_h^\star u_I^j,\xi)+a_j(r_u^j-P_h^\star r_u^j,\xi) \\
&=a_j(r_u^j-P_h^\star r_u^j,\xi).
\end{align*}
By using the property of the projection (\ref{projectionbound}) and (\ref{estimateremainder}), and the inverse inequality in (\ref{inverseineq}) for $\xi$, we have
\begin{align}
\sum_{j}a_j(\eta,\xi) \lesssim h^{2k+2} \|u\|_{k+2} + \|\xi\|^2.
\end{align}
Therefore, form (\ref{xietaestimate}), (\ref{proest}) and the stability result (\ref{stab}), we have
\begin{align}
\frac{1}{2}\frac{d}{dt}\|\xi\|^2 \lesssim h^{2k+2}\|u\|_{k+2}+\|\xi\|^2.
\end{align}
This together with the approximation results (\ref{proest}) and the initial datum
(\ref{initerr}), implies the desired error estimate (\ref{uniformest}).
\end{proof}

We summarize the theoretical findings and numerical findings in Table \ref{tab1d}.

\begin{table}[!ht]\centering
\caption{Summarization of the $L^2$ error accuracy for the 1D case.}
	\label{tab1d}
\begin{tabular}{|c|c|c|c|}
  \hline
   & mesh & $k$ is odd & $k$ is even \\ \hline
  \multirow{2}*{Numerically} & uniform & \multirow{2}*{$k$th} & ($k+1$)th \\ \cline{2-2} \cline{4-4}
   & non-uniform &  & $k$th \\ \hline
  \multirow{2}*{Theoretically} & uniform & \multirow{2}*{$k$th} & ($k+1$)th \\ \cline{2-2} \cline{4-4}
   &non-uniform &  & $k$th \\
  \hline
\end{tabular}
\end{table}

From Table \ref{tab1d},
we can see that our theoretical findings are sharp and consistent with the numerical results. We emphasize
that when $k$ is even, in order to produce the sub-optimal accuracy, we have designed a special regular
mesh sequence which is motivated by \cite{JGuzmanJSC2009}, see section \ref{Num}.

\section{Multi-dimensional problems}
\label{twodimension}
\setcounter{equation}{0}
\setcounter{figure}{0}
\setcounter{table}{0}

In this section, we consider the semidiscrete DG method with central fluxes
for multidimensional linear hyperbolic equations. Without loss of generality, we only study the two
dimensional problem; all the arguments we present in our analysis depends on the tensor product
structure of the mesh and the finite element space and can be easily extended to the more general
cases $d>2$. Hence, we consider the following two-dimensional problem
\begin{equation}
\label{2dlinear}
\left\{
\begin{aligned}
&u_{t}+u_{x}+u_y=0, \ (x,y,t) \in \Omega\times(0,T],\\
&u(x,y,0)=u_{0}(x,y), \ (x,y) \in \Omega.
\end{aligned}
\right.
\end{equation}
again with periodic boundary conditions. Without loss of generality, we assume $\Omega=[0,1]^2$. We
use the regular Cartesian mesh, $\left\{K_{i,j}=I_i\times J_j=[x_{\il},x_{\ir}]\times[y_{\jl},y_{\jr}]\right\}$, $i=1,\ldots,N_x$, $j=1,\ldots, N_y$. We denote $h_x^i=x_{\ir}-x_{\il}$, $h_y^j=y_{\jr}-y_{\jl}$ and $h =\max_{i,j}(h_x^i,h_y^j)$. Let $W_h:=\{v\in L^2(\Omega): v|_{K_{i,j}} \in \mathbb{Q}^k(K_{i,j}), \, \forall i,j\}$, where $\mathbb{Q}^k(K_{i,j})$ denotes the space of tensor-product polynomials of degrees at most $k$ in each variable defined on $K_{i,j}$.

The semidiscrete DG scheme with central fluxes is as follows.  We seek $u_h\in W_h$, such that for all test functions $v \in W_h$, and all $i,j$,
\begin{align}
\int_{K_{i,j}}(u_h)_t v\, dx dy =& \int_{K_{i,j}} (u_h v_x +u_h v_y) \,dx dy \nonumber \\
&- \int_{y_{\jl}}^{y_{\jr}} \left(\hat{u}_h(x_{\ir},y) v(x_{\ir}^-,y)-\hat{u}_h(x_{\il},y) v(x_{\il}^+,y)\right)\, dy\nonumber \\
&-\int_{x_{\il}}^{x_{\ir}} \left(\tilde{u}_h(x,y_{\jr}) v(x,y_{\jr}^-)-\tilde{u}_h(x,y_{\jl}) v(x,y_{\jl}^+)\right)\, dx \label{2dDGscheme}\\
&=:b_{i,j}(u_h,v)\label{2dDGop},
\end{align}
where
\begin{align}
\hat{u}_h(x_{\ir},y)=\frac{u_h(x_{\ir}^+,y)+u_h(x_{\ir}^-,y)}{2};\quad \tilde{u}_h(x,y_{\jr})=\frac{u_h(x,y_{\jr}^+)+u_h(x,y_{\jr}^-)}{2}.
\end{align}
For the initial data, we take $u_h(0)=Pu_0$, where $P$ is the $L^2$ projection into $W_h$, and we have {\rev\cite{ciarlet2002finite}}
\begin{align}\label{initerr2d}
\|u_0-Pu_0\| \lesssim h^{k+1}\|u_0\|_{k+1}.
\end{align}
{\rev{
We also have
\begin{align}
\label{energyconserving2d}
\sum_{i=1}^{N_x} \sum_{j=1}^{N_y} b_{i,j}(u_h,u_h) =0, \quad \forall u_h \in W_h.
\end{align}
}}
Thus we have the following energy conservative property
\begin{Pro}
\label{2dstab}
The numerical solution of (\ref{2dDGscheme}) satisfies
\begin{align}
\frac{1}{2}\frac{d}{dt}\|u_h\|^2=0 .
\end{align}
\end{Pro}

\subsection{A priori error estimates}

Let us now state our main result as a theorem, whose proof will be provided in the next subsection.
\begin{Thm}\label{thmmain2d}
Suppose $u_h$ is the numerical solution of the DG scheme (\ref{2dDGscheme}) for equation (\ref{2dlinear}) with a smooth initial condition $u(x,y,0)\in H^{k+2}(\Omega)$, and $u$ is the exact solution of (\ref{2dlinear}), then the approximation $u_h$ satisfies the following $L^2$ error estimate:
\begin{align}\label{nonuniformest2}
\|u(x,y,T)-u_h(x,y,T)\| \lesssim h^{k},
\end{align}
where $k$ is the degree of the piecewise tensor-product polynomials in the finite element spaces $W_h$.
Furthermore, when $k$ is even and the mesh is uniform, we have the optimal error estimate,
\begin{align}\label{uniformest2}
\|u(x,y,T)-u_h(x,y,T)\| \lesssim h^{k+1}.
\end{align}
\end{Thm}
\begin{Remark}
We note that the finite element space $V_h:=\{v\in L^2(\Omega): v|_{K_{i,j}} \in \mathbb{P}^k(K_{i,j}), \, \forall i,j\}$, where $\mathbb{P}^k(K_{i,j})$ denotes the space of polynomials of degrees at most $k$ defined on $K_{i,j}$, can also be taken as the
approximation space. But it only has the sub-optimal accuracy of order $k$ in the numerical examples, see section \ref{Num}.
Thus, here we only consider the tensor product space.
\end{Remark}
By the same arguments as in the one dimensional problem, we also have the error equation
\begin{align}
\label{errorequ2d}
\int_{K_{i,j}}(u-u_h)_t v \,dx dy -b_{i,j}(u-u_h,v)=0, \quad \forall v \in W_h, \, \forall i,j.
\end{align}
\subsection{Proof of the error estimates}
We divide the proof of Theorem \ref{thmmain2d} into several steps. First,
for non-uniform meshes, the proof of the sub-optimal error estimate is straightforward.
We just need to use the standard $L^2$ projection and follow the standard error estimates of DG methods
which is the same as in the one dimensional case. Thus next we only consider the uniform mesh case. In
order to prove the optimal error estimate when $k$ is even, we need to construct the special local
projection $\Pi_h^\star$. In addition, the optimal approximation properties of $\Pi_h^\star$ are
derived. The superconvergence results of the special projections would be given in the
subsection \ref{projectionpro}. Finally, we finish the proof of Theorem \ref{thmmain2d} in
subsection \ref{finishproof}.

\subsubsection{The special projection $\Pi_h^\star$}
\label{specialprojection2d}
Since our finite element space consists of piecewise $\mathbb{Q}^k$ polynomials, we use the tensor product
technique to construct the 2D projection. We define $\Pi_h^\star$ as the following projection into $W_h$. For each $K_{i,j}$,
\begin{subequations}\label{2dprojection}
\begin{equation}\label{2dph1}
\int_{K_{i,j}} \Pi_h^\star w(x,y) v(x,y)\, dxdy= \int_{K_{i,j}} w(x,y) v(x,y) \, dx dy,\quad \forall v\in \mathbb{Q}^{k-1}(K_{i,j}).
\end{equation}
\begin{equation}\label{2dph2}
\resizebox{\textwidth}{!}{$
\int_{I_i} \frac{\Pi_h^\star w(x,y_{\jr}^-)+\Pi_h^\star w(x,y_{\jl}^+)}{2} \varphi(x) \,dx=
\int_{I_i} \frac{ w(x,y_{\jr}^-)+ w(x,y_{\jl}^+)}{2} \varphi(x) \, dx, \quad \forall \varphi(x)\in \mathbb{P}^{k-1}(I_i)$}
\end{equation}
\begin{equation}\label{2dph3}
\resizebox{\textwidth}{!}{$
\int_{J_j} \frac{\Pi_h^\star w(x_{\ir}^-,y)+\Pi_h^\star w(x_{\il}^+,y)}{2} \varphi(y) \, dy =
\int_{J_j} \frac{w(x_{\ir}^-,y)+ w(x_{\il}^+,y)}{2} \varphi(y) \, dy, \quad \forall \varphi(y) \in \mathbb{P}^{k-1}(J_j)
$}
\end{equation}
\begin{align}\label{2dph4}
&\frac{1}{4}\left(\Pi_h^\star w(x_{\ir}^-,y_{\jr}^-)+\Pi_h^\star w(x_{\ir}^-,y_{\jl}^+)+\Pi_h^\star w(x_{\il}^+,y_{\jr}^-)+\Pi_h^\star w(x_{\il}^+,y_{\jl}^+)\right) \nonumber \\
&=\frac{1}{4}\left(w(x_{\ir}^-,y_{\jr}^-)+w(x_{\ir}^-,y_{\jl}^+)+\ w(x_{\il}^+,y_{\jr}^-)+ w(x_{\il}^+,y_{\jl}^+)\right).
\end{align}
\end{subequations}
Again, since the projection is local, we only consider the projection defined on the reference
cell $[-1,1]\times[-1,1]$. We establish the existence and uniqueness of the projection when $k$
is even in the following lemma
\begin{Lem}\label{2dprobound}
When $k$ is even, the projection $\Pi_h^\star$ defined by (\ref{2dprojection}) on the cell $[-1,1]\times[-1,1]$ exists and is unique for any $L^\infty$ function $w$, and the projection is bounded in the $L^\infty$ norm, i.e.
\begin{align}\label{2dprojectionbound}
\|\Pi_h^\star w\|_{\infty} \leq C(k) \|w\|_{\infty},
\end{align}
where $C(k)$ is a constant that only depends on $k$ but is independent of $w$.
\end{Lem}
\proof The proof of this lemma is given in the Appendix; see section \ref{proofoflem2dpro}.

Since the projection is a $k$-th degree polynomial preserving local projection, standard approximation theory \cite{ciarlet2002finite} implies, for a smooth function $w$,
\begin{align}\label{errorofpro2d}
\|w-\Pi_h^\star w\|_{L^2(K_{i,j})} \lesssim h^{k+1}\|w\|_{k+1,K_{i,j}}.
\end{align}
For the two dimensional space, for any $\omega_h\in W_h$, the following inequalities hold,
\begin{align}\label{inverseinequality}
\|\partial_x \omega_h\|\lesssim h^{-1} \|\omega_h\|, \quad \|\omega_h\|_{L^2(\partial{K_{i,j}})}\lesssim h^{-1/2} \|\omega_h\|, \quad \|\omega_h\|_{\infty} \lesssim h^{-1} \|\omega_h\|,
\end{align}
where $\partial{K_{i,j}}$ is the boundary of cell $K_{i,j}$.
\begin{Remark}
By similar arguments as in the one dimensional problem, we note that the projection $\Pi_h^\star$ is not well defined when $k$ is odd.
\end{Remark}

\subsubsection{Properties of the projection $\Pi_h^\star$}
\label{projectionpro}
{\rev{By the similar arguments in the one dimensional case}}, we have the following lemma:
\begin{Lem}
\label{projectionlem}
Assume that $u=x^{k+1}$ or $y^{k+1}$. Let $u_{i,j}=\Pi_h^\star u|_{K_{i,j}}$. If $h_x^{i-1}=h_x^{i}=h_x^{i+1}=h_x$ and $h_y^{j-1}=h_y^{j}=h_y^{j+1}=h_y$, then $\forall (x,y)\in K_{i,j}$, we have following relationship:
\begin{align}\label{relapro}
&u(x-h_x,y)-u_{i-1,j}(x-h_x,y)=u(x,y)-u_{i,j}(x,y)=u(x+h_x)-u_{i+1,j}(x+h_x,y) \nonumber \\
&=u(x,y+h_y)-u_{i,j+1}(x,y+h_y)=u(x,y-h_y)-u_{i,j-1}(x,y-h_y).
\end{align}
\end{Lem}
Similar to the one dimensional case, we also have the following superconvergence result.
\begin{Pro}
\label{2dprosuper}
For a given index $(i,j)$, suppose that $u$ is a $(k+1)$th degree polynomial function in $\mathbb{P}^{k+1}(D_{i,j})$, where $D_{i,j}=K_{i-1,j} \bigcup K_{i+1,j} \bigcup K_{i,j} \bigcup K_{i,j-1} \bigcup K_{i,j+1}$. If $h_x^{i-1}=h_x^{i}=h_x^{i+1}$ and $h_y^{j-1}=h_y^{j}=h_y^{j+1}$, then we have
\begin{align}\label{2dprosuperequ}
b_{i,j}(\Pi_h^{\star} u,v)=b_{i,j}(u,v) \quad \forall v \in \mathbb{Q}^{k}(K_{i,j}),
\end{align}
where $b_{i,j}(\cdot,\cdot)$ is defined by (\ref{2dDGop}).
\end{Pro}
\proof We provide the proof of this Proposition in the Appendix; see section \ref{proofofprop2d}.

\subsubsection{Proof of Theorem \ref{thmmain2d}}
\label{finishproof}
Let
\begin{align}
\xi=u_h-\Pi_h^\star u; \quad \eta = u-\Pi_h^\star u.
\end{align}
From (\ref{errorequ2d}), we obtain
\begin{align}\label{xietaequ}
\int_{K_{i,j}}(\xi)_t v\, dx dy-b_{i,j}(\xi,v)=\int_{K_{i,j}}(\eta)_t v\, dx dy-b_{i,j}(\eta,v), \quad \forall v \in \mathbb{Q}^k(K_{i,j}).
\end{align}
Take $v=\xi\, {\rev{\in W_h}}$, for the left hand side of (\ref{xietaequ}), we use {\rev{(\ref{energyconserving2d})}} to obtain
\begin{align}\label{xistab}
\sum_{i,j}\int_{K_{i,j}}(\xi)_t \xi\, dx dy-b_{i,j}(\xi,\xi)=\frac{1}{2}\frac{d}{dt}\|\xi\|^2.
\end{align}
For each element $K_{i,j}$, we consider the Taylor expansion of $u$ around $(x_i,y_j)$:
\begin{align}
u=Tu+Ru,
\end{align}
where
\begin{align*}
Tu=\sum_{l=0}^{k+1}\sum_{m=0}^{l}\frac{1}{(l-m)!m!}\frac{\partial^l u(x_i,y_j)}{\partial x^{l-m}\partial y^{m}}(x-x_i)^{l-m}(y-y_j)^{m},
\end{align*}
\begin{align*}
Ru=(k+2)\sum_{m=0}^{k+2}\frac{(x-x_i)^{k+2-m}(y-y_j)^{m}}{(k+2-m)!m!}\int_{0}^{1}(1-s)\frac{\partial^{k+2}u(x_i^s,y_j^s)}{\partial x^{k+2-m}\partial y^{m}}\ ds.
\end{align*}
with $x_i^s=x_i+s(x-x_i)$, $y_j^s=y_j+s(y-y_j)$. Clearly, $Tu \in \mathbb{P}^{k+1}(D_{i,j})$. By the linearity of the projection, and from (\ref{2dprosuperequ}), we then get
\begin{align}
b_{i,j}(\eta,v)&=b_{i,j}(Tu-\Pi_h^\star Tu,v)+b_{i,j}(Ru-\Pi_h^\star Ru,v)\nonumber \\
&=b_{i,j}(Ru-\Pi_h^\star Ru,v).
\end{align}
Again recalling the Bramble-Hilbert lemma in \cite{ciarlet2002finite}
, we have
\begin{align}\label{approxRu}
\|Ru\|_{L^{\infty}(D_{i,j})}\leq C h^{k+1}|u|_{H^{k+2}(D_{i,j})}.
\end{align}
Thus, this together with the standard approximate proposition of the projection (\ref{errorofpro2d}), and
the inverse inequality in (\ref{inverseinequality}) for $\xi$, we have
\begin{align}\label{etaest}
\sum_{i,j}b_{i,j}(\eta,\xi)\lesssim h^{2k+2}\|u\|_{k+2}^2+\|\xi\|^2.
\end{align}
From (\ref{xistab}), (\ref{etaest}) and (\ref{xietaequ}), we have
\begin{align}
\frac{1}{2}\frac{d}{dt}\|\xi\|^2 \lesssim h^{2k+2}\|u\|_{k+2}^2 +\|\xi\|^2.
\end{align}
This together with the approximation results (\ref{errorofpro2d}) and the initial discretization (\ref{initerr2d}), implies the desired error estimate (\ref{uniformest2}).

To end this section, we summarize our theoretical findings and numerical findings for
the 2D problem in Table \ref{tab2d}. Again our theoretical proof is sharp and consistent with the numerical results.
\begin{table}[!ht]\centering
\caption{Summarization of the $L^2$ error accuracy for the 2D case.}
	\label{tab2d}
\begin{tabular}{|c|c|c|c|c|}
  \hline
  &  & mesh & $k$ is odd & $k$ is even \\ \cline{2-5}
  \multirow{4}*{$\mathbb{Q}^k$-space}&\multirow{2}*{Numerically} & uniform & \multirow{2}*{$k$th} & ($k+1$)th \\ \cline{3-3} \cline{5-5}
  & & non-uniform &  & $k$th \\ \cline{2-5}
  &\multirow{2}*{Theoretically} & uniform & \multirow{2}*{$k$th} & ($k+1$)th \\ \cline{3-3} \cline{5-5}
  & &non-uniform &  & $k$th \\
  \hline
  $\mathbb{P}^k$-space&Numerically/Theoretically & uniform/nonuniform & $k$th & $k$th \\ \hline
\end{tabular}
\end{table}
\section{Numerical examples}
\label{Num}
\setcounter{equation}{0}
\setcounter{figure}{0}
\setcounter{table}{0}

In this section, we present some numerical examples to verify our theoretical findings. In our numerical experiments,
we present the $E_2$, $E_A$, and $E_f$ errors, respectively. They are defined by
\begin{align}
E_2=&\|u-u_h\|.\\
E_A=&\left\{\begin{array}{c}\left(\frac{1}{N}\sum_{j=1}^{N}(\frac{1}{h_j}\int_{I_j}(u-u_h)\, dx)^2\right)^{\frac{1}{2}},\quad \text{ for one dimension,}\\
\left(\frac{1}{N_x N_y}\sum_{j=1}^{N_y}\sum_{i=1}^{N_x}(\frac{1}{h_x^i h_y^j}\int_{K_{i,j}}(u-u_h)\, dx dy)^2\right)^{\frac{1}{2}},\quad \text{ for two dimensions.}
\end{array}\right.\\
E_f=&\left(\frac{1}{N}\sum_{j=1}^{N}(u_{\jr}-\{u_h\}_{\jr})^2\right)^{\frac{1}{2}}.
\end{align}

\begin{Ex}
	\label{example1}
We consider the linear hyperbolic equation  with periodic boundary
condition:
	\begin{align}
\left\{\begin{array}{l}
	u_t+u_x=0,\quad (x,t)\in [0,2\pi]\times(0,T),\\
	u(x,0)=\exp(\sin(x)),\\
	u(0,t)=u(2\pi,t).
\end{array}\right.
	\end{align}
\end{Ex}
The exact solution to this problem is
\begin{equation}
u(x,t)=\exp(\sin(x-t)).
\end{equation}
We use two kinds of non-uniform meshes.  The first one is the non-uniform mesh with 30\% random perturbation
from $N$ uniform cells on $[0,2\pi]$, and the other mesh is constructed as follows.  Let
$\tilde{x}_{\jr}=jh$ for $j=0,\ldots,N$ where $h=\frac{2\pi}{N}$ and $\tilde{x}_{N+\frac{1}{2}}=1$, then we define the nodes of our mesh as follows
\begin{align*}
x_{2j-\frac{1}{2}}=\tilde{x}_{2j-\frac{1}{2}}+\alpha h,\quad j=1,\ldots, \left\lfloor\frac{N}{2}\right\rfloor.
\end{align*}
where $\lfloor m \rfloor$ denotes the maximal integer no more than $m$. Here the
parameter $\alpha$ satisfies $-1<\alpha<1$. For example, if $\alpha=0$ then the
resulting mesh is uniform.

We set the number of subintervals, $N=2^i\times10,i=0,\ldots,9$, in our experiments. We
use the DG scheme (\ref{DGscheme}) with central fluxes using $\mathbb{P}^k$ polynomials with $k=0,2,4$.
The initial datum is obtained by the standard $L^2$ projection. To reduce the time discretization error,
the seventh-order strong stability-preserving Runge-Kutta method \cite{gottlieb2001strong} with the time step $\Delta t= 0.01h$ is used.
The errors and corresponding convergence rates for {\rev{ the special nonuniform mesh with $\alpha=0.1$, the uniform mesh,
and random perturbation mesh are separately listed in the Tables \ref{tab1d1}-\ref{tab1d3}}}. Since the convergence rates have oscillations,
especially for $E_A$ and $E_f$, we have used the least square method to fit the convergence orders of the errors,
denoted by ``LS order" in the tables. We can find that $E_2$ only has $k$-th order accuracy, but $E_A$ and $E_f$ have ($k+1$)-th order
convergence for $k=2,4$, when the parameter of mesh $\alpha=0.1$. For the uniform mesh, i.e., $\alpha=0$, we observe the ($k+1$)-th
optimal convergence rates.  We can also find the convergence rates of the $L^2$ errors to be around $k+\frac{1}{2}$ for
the randomly perturbed meshes.

\begin{table}[!ht]\centering
	\caption{The errors and corresponding convergence rates for
the DG with $k=0,2,4$ in 1D. The terminal time $T=1$ and the parameter of the mesh $\alpha=0.1$.}
	\label{tab1d1}
	\begin{tabular}{|c|c|cc|cc|cc|}
		\hline
		\multirow{12}*{$k=0$}
		&$N$   & $E_2$    &Rate    &$E_A$ & Rate &$E_f$ & Rate\\ \cline{2-8}
&   10&    5.14E-01&     --&    1.23E-01&--&    1.15E-01&--\\ \cline{2-8}
&   20&    2.75E-01&     0.90&    7.32E-02&     0.75&    3.20E-02&     1.85\\ \cline{2-8}
&   40&    2.02E-01&     0.44&    6.98E-02&     0.07&    7.66E-03&     2.06\\ \cline{2-8}
&   80&    1.82E-01&     0.15&    7.01E-02&    -0.01&    5.75E-03&     0.41\\ \cline{2-8}
&  160&    1.77E-01&     0.04&    7.03E-02&    -0.00&    6.45E-03&    -0.16\\ \cline{2-8}
&  320&    1.75E-01&     0.01&    7.03E-02&    -0.00&    6.68E-03&    -0.05\\ \cline{2-8}
&  640&    1.75E-01&     0.00&    7.03E-02&    -0.00&    6.75E-03&    -0.01\\ \cline{2-8}
& 1280&    1.75E-01&     0.00&    7.04E-02&    -0.00&    6.76E-03&    -0.00\\ \cline{2-8}
& 2560&    1.75E-01&     0.00&    7.04E-02&    -0.00&    6.77E-03&    -0.00\\ \cline{2-8}
& 5120&    1.75E-01&     0.00&    7.04E-02&    -0.00&    6.77E-03&    -0.00\\ \cline{2-8}
&LS order&          &   0.12&               &   0.05&           &       0.32 \\ \hline
		\multirow{11}*{$k=2$}
&   10&    9.30E-03&     --&    1.09E-03&--&    2.07E-03&--\\ \cline{2-8}
&   20&    7.82E-04&     3.57&    8.21E-05&     3.73&    2.20E-04&     3.23\\ \cline{2-8}
&   40&    1.33E-04&     2.55&    9.77E-06&     3.07&    2.10E-05&     3.39\\ \cline{2-8}
&   80&    2.00E-05&     2.73&    9.26E-07&     3.40&    2.19E-06&     3.26\\ \cline{2-8}
&  160&    4.21E-06&     2.25&    1.21E-07&     2.94&    2.36E-07&     3.22\\ \cline{2-8}
&  320&    9.99E-07&     2.07&    2.10E-08&     2.52&    1.48E-08&     4.00\\ \cline{2-8}
&  640&    2.46E-07&     2.02&    1.98E-09&     3.41&    3.35E-09&     2.14\\ \cline{2-8}
& 1280&    6.13E-08&     2.01&    3.37E-10&     2.55&    8.36E-11&     5.32\\ \cline{2-8}
& 2560&    1.53E-08&     2.00&    4.64E-12&     6.18&    7.30E-11&     0.20\\ \cline{2-8}
& 5120&    3.83E-09&     2.00&    1.14E-12&     2.03&    8.80E-12&     3.05\\ \cline{2-8}
&LS order&          &   2.28&               &   3.27&           &       3.17 \\ \hline
\multirow{11}*{$k=4$}
&   10&    1.21E-04&    --&    2.18E-06&--&    1.61E-05&--\\ \cline{2-8}
&   20&    1.62E-06&     6.22&    6.45E-08&     5.08&    6.56E-07&     4.62\\ \cline{2-8}
&   40&    9.60E-08&     4.08&    2.82E-09&     4.51&    1.41E-08&     5.53\\ \cline{2-8}
&   80&    5.28E-09&     4.19&    1.50E-10&     4.24&    4.13E-10&     5.10\\ \cline{2-8}
&  160&    3.22E-10&     4.04&    1.72E-12&     6.44&    1.64E-11&     4.65\\ \cline{2-8}
&  320&    1.99E-11&     4.02&    8.17E-14&     4.40&    2.34E-13&     6.13\\ \cline{2-8}
&  640&    1.24E-12&     4.00&    3.41E-16&     7.90&    1.81E-14&     3.69\\ \cline{2-8}
& 1280&    7.75E-14&     4.00&    3.18E-17&     3.42&    4.93E-16&     5.19\\ \cline{2-8}
& 2560&    4.85E-15&     4.00&    1.27E-18&     4.64&    1.59E-17&     4.95\\ \cline{2-8}
& 5120&    3.03E-16&     4.00&    5.35E-20&     4.57&    4.69E-19&     5.08\\ \cline{2-8}
&LS order&          &   4.16&               &   5.14&           &       5.00 \\ \hline
	\end{tabular}
\end{table}

\begin{table}[!ht]\centering
	\caption{The errors and corresponding convergence rates for
the DG with $k=0,2,4$ {\rev{using the uniform mesh}} in 1D. The terminal time $T=1$.}
	\label{tab1d2}
	\begin{tabular}{|c|c|cc|cc|cc|}
		\hline
		\multirow{12}*{$k=0$}
		&$N$   & $E_2$    &Rate    &$E_A$ & Rate &$E_f$ & Rate\\ \cline{2-8}
&   10&    4.82E-01&     --&    1.07E-01&--&    1.22E-01&--\\ \cline{2-8}
&   20&    2.16E-01&     1.16&    3.06E-02&     1.80&    3.53E-02&     1.79\\ \cline{2-8}
&   40&    1.03E-01&     1.07&    7.91E-03&     1.95&    9.16E-03&     1.95\\ \cline{2-8}
&   80&    5.09E-02&     1.02&    1.99E-03&     1.99&    2.31E-03&     1.99\\ \cline{2-8}
&  160&    2.54E-02&     1.01&    4.99E-04&     2.00&    5.79E-04&     2.00\\ \cline{2-8}
&  320&    1.27E-02&     1.00&    1.25E-04&     2.00&    1.45E-04&     2.00\\ \cline{2-8}
&  640&    6.34E-03&     1.00&    3.12E-05&     2.00&    3.62E-05&     2.00\\ \cline{2-8}
& 1280&    3.17E-03&     1.00&    7.81E-06&     2.00&    9.06E-06&     2.00\\ \cline{2-8}
& 2560&    1.58E-03&     1.00&    1.95E-06&     2.00&    2.27E-06&     2.00\\ \cline{2-8}
& 5120&    7.92E-04&     1.00&    4.88E-07&     2.00&    5.66E-07&     2.00\\ \cline{2-8}
&LS order&          &   1.02&               &   1.98&           &      1.98 \\ \hline
		\multirow{11}*{$k=2$}
&   10&    9.11E-03&     --&    1.27E-03&--&    2.50E-03&--\\ \cline{2-8}
&   20&    5.47E-04&     4.06&    1.78E-05&     6.15&    8.32E-05&     4.91\\ \cline{2-8}
&   40&    6.12E-05&     3.16&    5.25E-07&     5.08&    3.13E-06&     4.73\\ \cline{2-8}
&   80&    7.52E-06&     3.03&    1.23E-08&     5.42&    3.41E-07&     3.20\\ \cline{2-8}
&  160&    9.32E-07&     3.01&    3.29E-10&     5.22&    2.44E-08&     3.81\\ \cline{2-8}
&  320&    1.16E-07&     3.00&    1.45E-11&     4.50&    3.58E-10&     6.09\\ \cline{2-8}
&  640&    1.45E-08&     3.00&    1.84E-13&     6.31&    1.27E-10&     1.50\\ \cline{2-8}
& 1280&    1.82E-09&     3.00&    1.08E-14&     4.08&    5.42E-12&     4.55\\ \cline{2-8}
& 2560&    2.27E-10&     3.00&    4.23E-16&     4.68&    1.35E-13&     5.32\\ \cline{2-8}
& 5120&    2.84E-11&     3.00&    6.77E-18&     5.97&    2.95E-14&     2.20\\ \cline{2-8}
&LS order&          &   3.08&               &   5.18&           &      4.04 \\ \hline
\multirow{11}*{$k=4$}
&   10&    1.18E-04&    --&    1.56E-06&--&    2.03E-05&--\\ \cline{2-8}
&   20&    1.03E-06&     6.84&    2.28E-08&     6.09&    3.13E-07&     6.02\\ \cline{2-8}
&   40&    2.76E-08&     5.22&    1.27E-10&     7.49&    5.78E-09&     5.76\\ \cline{2-8}
&   80&    8.11E-10&     5.09&    1.83E-12&     6.11&    8.19E-11&     6.14\\ \cline{2-8}
&  160&    2.49E-11&     5.03&    4.99E-15&     8.52&    1.94E-12&     5.40\\ \cline{2-8}
&  320&    7.78E-13&     5.00&    2.19E-17&     7.83&    2.71E-14&     6.16\\ \cline{2-8}
&  640&    2.43E-14&     5.00&    1.71E-19&     7.00&    3.79E-16&     6.16\\ \cline{2-8}
& 1280&    7.59E-16&     5.00&    5.35E-21&     5.00&    2.06E-18&     7.53\\ \cline{2-8}
& 2560&    2.37E-17&     5.00&    2.05E-23&     8.03&    1.20E-19&     4.10\\ \cline{2-8}
& 5120&    7.41E-19&     5.00&    6.25E-26&     8.35&    1.50E-21&     6.32\\ \cline{2-8}
&LS order&          &   5.14&               &   7.15&           &       5.98 \\ \hline
	\end{tabular}
\end{table}

\begin{table}[!ht]\centering
	\caption{The errors and corresponding convergence rates for
the DG with $k=0,2,4$ using random perturbation mesh in 1D. The terminal time $T=1$.}
	\label{tab1d3}
	\begin{tabular}{|c|c|cc|cc|cc|}
		\hline
		\multirow{11}*{$k=0$}
		&$N$   & $E_2$    &Rate    &$E_A$ & Rate &$E_f$ & Rate\\ \cline{2-8}
&   10&    6.99E-01&     --&    2.30E-01&--&    1.48E-01&--\\ \cline{2-8}
&   20&    9.78E-01&    -0.48&    4.30E-01&    -0.90&    1.82E-01&    -0.30\\ \cline{2-8}
&   40&    3.06E-01&     1.68&    1.16E-01&     1.89&    4.26E-02&     2.10\\ \cline{2-8}
&   80&    2.90E-01&     0.08&    1.12E-01&     0.05&    2.54E-02&     0.75\\ \cline{2-8}
&  160&    1.63E-01&     0.83&    6.49E-02&     0.79&    1.48E-02&     0.78\\ \cline{2-8}
&  320&    1.74E-01&    -0.10&    6.99E-02&    -0.11&    1.39E-02&     0.09\\ \cline{2-8}
&  640&    8.47E-02&     1.04&    3.37E-02&     1.05&    5.15E-03&     1.43\\ \cline{2-8}
& 1280&    6.77E-02&     0.32&    2.69E-02&     0.33&    2.96E-03&     0.80\\ \cline{2-8}
& 2560&    5.35E-02&     0.34&    2.14E-02&     0.33&    2.08E-03&     0.51\\ \cline{2-8}
& 5120&    2.71E-02&     0.98&    1.08E-02&     0.98&    1.01E-03&     1.04\\ \cline{2-8}
& LS order&    &     0.53&            &     0.52&            &     0.83\\ \hline
		\multirow{10}*{$k=2$}
&   10&    2.73E-02&     --&    1.89E-03&--&    6.94E-03&--\\ \cline{2-8}
&   20&    7.18E-03&     1.93&    1.72E-04&     3.46&    4.07E-04&     4.09\\ \cline{2-8}
&   40&    1.05E-03&     2.78&    1.46E-05&     3.55&    3.14E-05&     3.70\\ \cline{2-8}
&   80&    1.10E-04&     3.24&    4.44E-06&     1.72&    7.56E-06&     2.05\\ \cline{2-8}
&  160&    3.27E-05&     1.75&    4.76E-07&     3.22&    8.11E-07&     3.22\\ \cline{2-8}
&  320&    4.21E-06&     2.96&    7.29E-08&     2.71&    1.16E-07&     2.81\\ \cline{2-8}
&  640&    1.00E-06&     2.07&    1.07E-08&     2.77&    1.67E-08&     2.80\\ \cline{2-8}
& 1280&    8.97E-08&     3.48&    1.52E-09&     2.81&    2.12E-09&     2.98\\ \cline{2-8}
& 2560&    3.12E-08&     1.52&    1.99E-10&     2.93&    2.52E-10&     3.07\\ \cline{2-8}
& 5120&    6.36E-09&     2.29&    2.80E-11&     2.83&    3.52E-11&     2.84\\ \cline{2-8}
& LS order&    &     2.51&            &     2.83&            &     2.98\\ \hline
\multirow{11}*{$k=4$}
&   10&    2.28E-04&    --&    7.01E-06&--&    3.73E-05&--\\ \cline{2-8}
&   20&    7.81E-06&     4.87&    3.74E-07&     4.23&    1.92E-06&     4.28\\ \cline{2-8}
&   40&    2.72E-07&     4.84&    6.81E-09&     5.78&    3.37E-08&     5.83\\ \cline{2-8}
&   80&    1.53E-08&     4.15&    3.41E-10&     4.32&    1.71E-09&     4.31\\ \cline{2-8}
&  160&    1.57E-09&     3.29&    1.34E-11&     4.66&    6.94E-11&     4.62\\ \cline{2-8}
&  320&    7.00E-11&     4.48&    3.46E-13&     5.28&    1.38E-12&     5.65\\ \cline{2-8}
&  640&    5.60E-12&     3.65&    1.43E-14&     4.59&    5.33E-14&     4.70\\ \cline{2-8}
& 1280&    7.04E-14&     6.31&    4.29E-16&     5.06&    1.57E-15&     5.09\\ \cline{2-8}
& 2560&    1.88E-15&     5.23&    1.70E-17&     4.66&    5.10E-17&     4.94\\ \cline{2-8}
& 5120&    1.48E-16&     3.66&    5.87E-19&     4.85&    1.67E-18&     4.93\\ \cline{2-8}
& LS order&    &         4.46&            &     4.85&            &     4.95\\ \hline
	\end{tabular}
\end{table}

In two dimensions, we consider the following problem.

\begin{Ex}
\label{example2}
We solve the following linear hyperbolic equation with periodic boundary condition:
\begin{align}
\left\{\begin{array}{l}
u_t+u_x+u_y=0, \quad (x,y,t)\in [0,2\pi]^2\times(0,T),\\
u(x,y,0)=\sin(x+y).
\end{array}\right.
\end{align}
\end{Ex}
The exact solution to this problem is
\begin{align}
u(x,y,t)=\sin(x+y-2t).
\end{align}
In each dimension, we apply the same partition as in the one-dimensional case.  We choose the parameters $\alpha=0.3$ (see Fig \ref{fig1}). The tensor product space $\mathbb{Q}^k$ or the piecewise $k$th polynomial $\mathbb{P}^{k}$ is taken as the approximation space. We test the DG scheme with the central flux, and take the terminal time $T=1$.
When $\mathbb{Q}^k$ elements are used and $k=0,2$, {\rev{for the special nonuniform mesh with $\alpha=0.3$}}, the sub-optimal $k$th convergence rates can be observed which are listed in
Table \ref{tab2d1}. For the uniform mesh, i.e., $\alpha=0$, the scheme has ($k+1$)th optimal convergence orders, see Table {\ref{tab2d2}}.
However, for $\mathbb{P}^k$ finite element space, it only has $k$th suboptimal convergence rates no matter whether $k$ is even or odd, see Table \ref{tab2d3}.

\begin{table}[!ht]\centering
	\caption{The errors and corresponding convergence rates for
the DG with $k=0,2$ in 2D. The terminal time $T=1$ and the parameter of the mesh $\alpha=0.3$.}
	\label{tab2d1}
	\begin{tabular}{|c|c|cc|cc|}
		\hline
\multirow{8}*{$\mathbb{Q}^0$}
&$N\times N$         &     $E_2$  &Rate     &$E_A$       & Rate \\ \cline{2-6}
& $    4\times    4$ &    3.67E+00&    --&    3.56E-01&--\\ \cline{2-6}
& $    8\times    8$ &    2.11E+00&     0.80&    2.34E-01&     0.60\\ \cline{2-6}
& $   16\times   16$ &    1.74E+00&     0.28&    2.82E-01&    -0.27\\ \cline{2-6}
& $   32\times   32$ &    1.67E+00&     0.06&    3.02E-01&    -0.10\\ \cline{2-6}
& $   64\times   64$ &    1.66E+00&     0.01&    3.07E-01&    -0.02\\ \cline{2-6}
& $  128\times  128$ &    1.65E+00&     0.00&    3.08E-01&    -0.01\\ \cline{2-6}
& LS order           &            &     0.20&            &     -0.01 \\ \hline
\multirow{7}*{$\mathbb{Q}^2$}
& $    5\times    5$ &    1.43E-01&       --&    4.08E-03&--\\ \cline{2-6}
& $    9\times    9$ &    4.28E-02&     1.74&    1.64E-03&     1.32\\ \cline{2-6}
& $   17\times   17$ &    1.24E-02&     1.78&    3.06E-04&     2.42\\ \cline{2-6}
& $   33\times   33$ &    3.33E-03&     1.90&    2.11E-05&     3.85\\ \cline{2-6}
& $   65\times   65$ &    8.67E-04&     1.94&    3.12E-06&     2.76\\ \cline{2-6}
& $  129\times  129$ &    2.21E-04&     1.97&    6.57E-07&     2.25\\ \cline{2-6}
& LS order           &            &     1.99&            &     2.85 \\ \hline
	\end{tabular}
\end{table}

\begin{table}[!ht]\centering
	\caption{The errors and corresponding convergence rates for
the DG with $k=0,2$ {\rev{using the uniform mesh}} in 2D. The terminal time $T=1$.}
	\label{tab2d2}
	\begin{tabular}{|c|c|cc|cc|}
		\hline
\multirow{8}*{$\mathbb{Q}^0$}
&$N\times N$         &     $E_2$  &Rate     &$E_A$       & Rate \\ \cline{2-6}
& $    4\times    4$ &    3.65E+00&    --&    4.07E-01&--\\ \cline{2-6}
& $    8\times    8$ &    1.63E+00&     1.17&    1.34E-01&     1.61\\ \cline{2-6}
& $   16\times   16$ &    7.43E-01&     1.13&    3.56E-02&     1.91\\ \cline{2-6}
& $   32\times   32$ &    3.60E-01&     1.04&    9.04E-03&     1.98\\ \cline{2-6}
& $   64\times   64$ &    1.79E-01&     1.01&    2.27E-03&     1.99\\ \cline{2-6}
& $  128\times  128$ &    8.91E-02&     1.00&    5.68E-04&     2.00\\ \cline{2-6}
& LS order           &            &     1.07&            &     1.92 \\ \hline
\multirow{7}*{$\mathbb{Q}^2$}
& $    4\times    4$ &    1.99E-01&     --&    8.35E-03&--\\ \cline{2-6}
& $    8\times    8$ &    1.27E-02&     3.97&    7.97E-05&     6.71\\ \cline{2-6}
& $   16\times   16$ &    1.21E-03&     3.39&    3.31E-06&     4.59\\ \cline{2-6}
& $   32\times   32$ &    1.51E-04&     2.99&    1.85E-07&     4.16\\ \cline{2-6}
& $   64\times   64$ &    1.88E-05&     3.01&    1.87E-09&     6.63\\ \cline{2-6}
& $  128\times  128$ &    2.34E-06&     3.01&    1.17E-10&     4.00\\ \cline{2-6}
& LS order           &            &     3.23&            &     5.16 \\ \hline
	\end{tabular}
\end{table}

\begin{table}[!ht]\centering
	\caption{The errors and corresponding convergence rates for
the DG with $k=0,2$ using random perturbation mesh in 2D. The terminal time $T=1$.}
	\label{tab2d3}
	\begin{tabular}{|c|c|cc|cc|}
		\hline
\multirow{8}*{$\mathbb{Q}^0$}
&$N\times N$         &     $E_2$  &Rate     &$E_A$       & Rate \\ \cline{2-6}
& $    4\times    4$ &    3.76E+00&    --&    3.80E-01&--\\ \cline{2-6}
& $    8\times    8$ &    1.78E+00&     1.08&    1.55E-01&     1.30\\ \cline{2-6}
& $   16\times   16$ &    1.23E+00&     0.53&    1.56E-01&    -0.01\\ \cline{2-6}
& $   32\times   32$ &    8.83E-01&     0.48&    1.31E-01&     0.25\\ \cline{2-6}
& $   64\times   64$ &    7.05E-01&     0.33&    1.07E-01&     0.30\\ \cline{2-6}
& $  128\times  128$ &    5.63E-01&     0.32&    8.85E-02&     0.28\\ \cline{2-6}
& LS order           &            &     0.52&            &     0.35 \\ \hline
\multirow{7}*{$\mathbb{Q}^2$}
& $    4\times    4$ &    4.62E-01&     --&    1.17E-02&--\\ \cline{2-6}
& $    8\times    8$ &    3.22E-02&     3.84&    1.31E-03&     3.15\\ \cline{2-6}
& $   16\times   16$ &    1.04E-02&     1.64&    1.98E-04&     2.73\\ \cline{2-6}
& $   32\times   32$ &    2.06E-03&     2.33&    2.23E-05&     3.15\\ \cline{2-6}
& $   64\times   64$ &    3.34E-04&     2.63&    2.93E-06&     2.93\\ \cline{2-6}
& $  128\times  128$ &    3.56E-05&     3.23&    4.41E-07&     2.74\\ \cline{2-6}
& LS order           &            &     2.58&            &     2.91 \\ \hline
	\end{tabular}
\end{table}

\begin{table}[!ht]\centering
	\caption{The errors and corresponding convergence rates for
the DG with using $\mathbb{P}^k$ finite element space in 2D. The terminal time $T=1$ and the parameter of the mesh $\alpha=0$.}
	\label{tab2}
	\begin{tabular}{|c|c|cc|cc|}
		\hline
\multirow{7}*{$\mathbb{P}^1$}
&$N\times N$         &     $E_2$  &Rate     &$E_A$       & Rate \\ \cline{2-6}
& $    4\times    4$ &    1.51E+00&    --&    7.53E-02&--\\ \cline{2-6}
& $    8\times    8$ &    6.88E-01&     1.13&    2.12E-02&     1.83\\ \cline{2-6}
& $   16\times   16$ &    3.30E-01&     1.06&    5.77E-03&     1.88\\ \cline{2-6}
& $   32\times   32$ &    1.63E-01&     1.02&    1.48E-03&     1.96\\ \cline{2-6}
& $   64\times   64$ &    8.11E-02&     1.01&    3.72E-04&     1.99\\ \cline{2-6}
& $  128\times  128$ &    4.05E-02&     1.00&    9.32E-05&     2.00\\ \hline
\multirow{6}*{$\mathbb{P}^2$}
& $    4\times    4$ &    3.81E-01&     --&    7.83E-03&--\\ \cline{2-6}
& $    8\times    8$ &    6.31E-02&     2.59&    1.82E-03&     2.11\\ \cline{2-6}
& $   16\times   16$ &    2.01E-02&     1.65&    7.97E-05&     4.51\\ \cline{2-6}
& $   32\times   32$ &    5.31E-03&     1.92&    5.05E-06&     3.98\\ \cline{2-6}
& $   64\times   64$ &    1.34E-03&     1.98&    3.29E-07&     3.94\\ \cline{2-6}
& $  128\times  128$ &    3.37E-04&     2.00&    2.04E-08&     4.01\\ \hline
\multirow{6}*{$\mathbb{P}^3$}
& $    4\times    4$ &    2.63E-01&     1.93&    1.30E-03&--\\ \cline{2-6}
& $    8\times    8$ &    7.84E-03&     5.07&    7.47E-05&     4.12\\ \cline{2-6}
& $   16\times   16$ &    3.83E-04&     4.35&    3.63E-07&     7.68\\ \cline{2-6}
& $   32\times   32$ &    3.39E-05&     3.50&    2.74E-08&     3.73\\ \cline{2-6}
& $   64\times   64$ &    3.70E-06&     3.19&    4.84E-10&     5.82\\ \cline{2-6}
& $  128\times  128$ &    4.46E-07&     3.05&    1.94E-11&     4.64\\ \cline{2-6}
& $  256\times  256$ &    5.56E-08&     3.01&    9.59E-13&     4.34\\ \hline
	\end{tabular}
\end{table}

\begin{figure}[h!]
  \centering
  \includegraphics[width=0.5\textwidth]{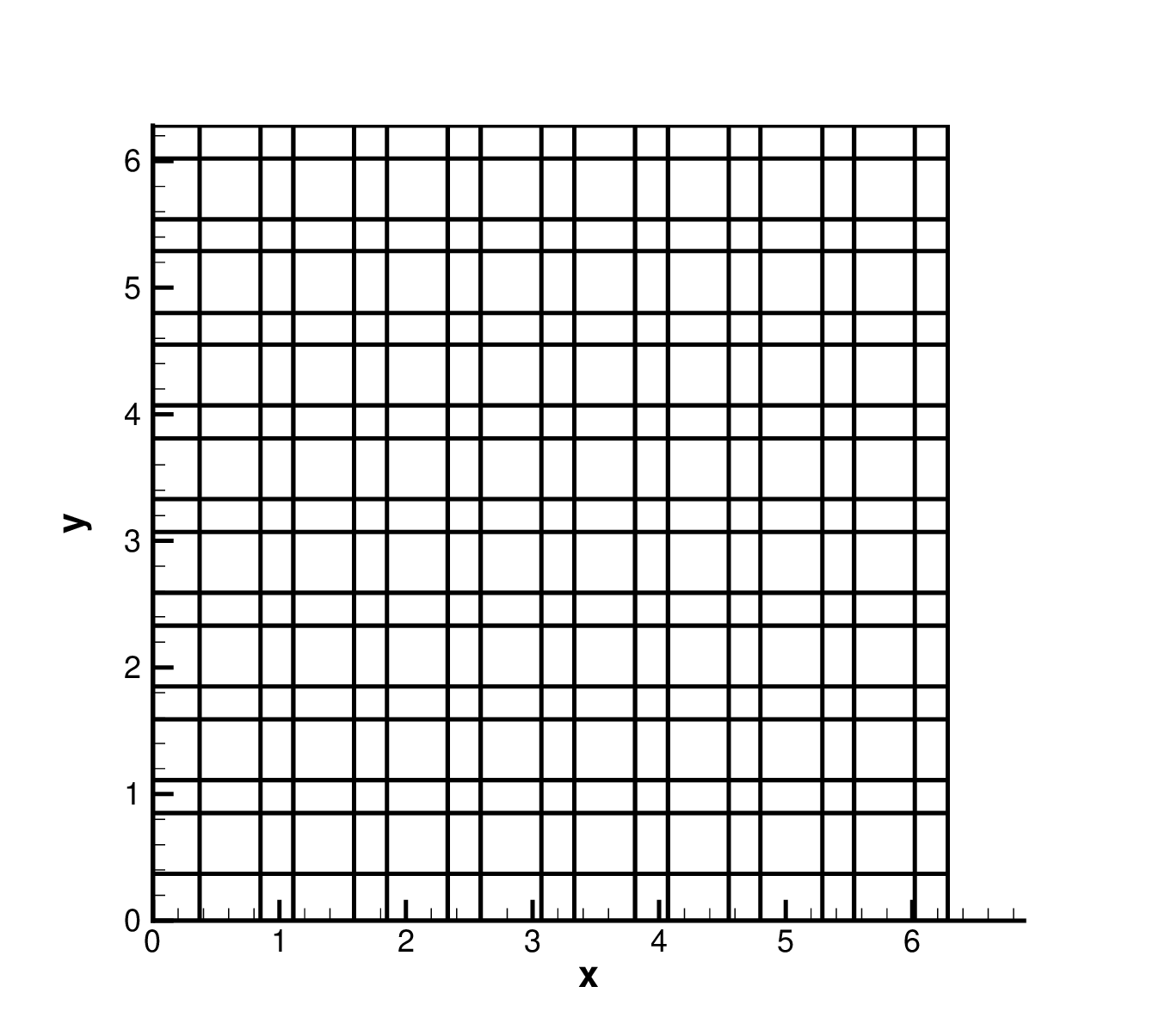}\\
  \caption{Example of non-uniform two dimensional mesh with $h=\frac{2\pi}{17}$ and $\alpha=0.3$}\label{fig1}
\end{figure}

\section{Concluding remarks}
\label{concluding}
In this paper, we have studied the error estimates of the DG methods for linear hyperbolic equations with central fluxes
when the degree of piecewise polynomial is even. Numerically, we provide a counter example to show that the scheme only
has the sub-optimal convergence rates for a particular regular non-uniform mesh sequence. Theoretically, we show that
the scheme does have the optimal accuracy for uniform meshes in both one and multi-dimensional problems. Our proof
does not have the constraint on the number of cells in the mesh.
In numerical experiments, we have also observed that the superconvergence results for the errors of the
cell averages and numerical fluxes. The theoretical proof of these superconvergence results would be interesting and
challenging for our future work.

\appendix
\section{Appendix: Proof of a few technical lemmas and Propositions.}
\label{appendix}
\setcounter{equation}{0}
In this appendix, we provide the proof of some of the technical lemmas and propositions in the error estimates.
\subsection{Proof of Lemma \ref{lemprojection}}
\label{proofoflemmaprojection}
\begin{proof}
Note that the procedure to find $P_h^\star w \in \mathbb{P}^{k}([-1,1])$ is to solve a linear system, so the existence and uniqueness are equivalent. Thus, we only prove the uniqueness of the projection $P_h^\star$. We set $w_I(x)=P_h^\star w(x)$ with $w(x)=0$ and would like to prove $w_I(x)\equiv0$. By the definition of the projection $P_h^\star w(x)$, then
\begin{align}
&\widetilde{P_h}(w_I;v)=-\int_{-1}^{1} w_I v_x \, dx +\frac{w_I(-1)+w_I(1)}{2}(v(1)-v(-1))=0, \quad \forall v\in \mathbb{P}^{k}([-1,1]), \label{equ1}\\
&\int_{-1}^{1}w_I\, dx=0. \label{equ2}
\end{align}
If $k=0$, the lemma obviously holds. If $k\geq 2$, we take $v=x$ in (\ref{equ1}), we have
\begin{align}
0=\widetilde{P_h}(w_I;x)&=-\int_{-1}^{1} w_I \, dx + (w_I(-1)+w_I(1)) \nonumber\\
&=(w_I(-1)+w_I(1)),\label{equ3}
\end{align}
here we have used (\ref{equ2}). Thus in (\ref{equ1}), we have
\begin{align*}
0=\widetilde{P_h}(w_I;v)=-\int_{-1}^{1} w_I v_x \, dx, \quad \forall v \in \mathbb{P}^k([-1,1]),
\end{align*}
which means that
\begin{align}
w_I=w_k L_k(x),
\end{align}
where $w_k$ is a constant, and $L_k(x)$ is the standard $k$-th degree
Legendre polynomial on the interval $[-1,1]$. By (\ref{equ3}) and the fact that $k$ is even, we obtain $w_k\equiv0$.
This finished the proof of uniqueness. For the proof of the bound (\ref{projectionbound}), the arguments are the
same as in \cite{liu2018cdg,liu2019pk}, hence we omit them here.
\end{proof}

\subsection{Proof of Lemma \ref{2dprobound}}
\label{proofoflem2dpro}
\begin{proof}
When $k=0$, the Lemma obviously holds true. For $k\geq2$, assume that $u\equiv 0$. From (\ref{2dph1}), we have $\Pi_h^\star u \bot \mathbb{Q}^{k-1}([-1,1]^2)$, thus we have the following expression of $\Pi_h^\star u$,
\begin{align}
\Pi_h^\star u=\sum_{m=0}^{k-1}\alpha_{k,m}L_k(x)L_m(y)+\sum_{m=0}^{k-1}\alpha_{m,k}L_{m}(x)L_{k}(y)+\alpha_{k,k}L_{k}(x)L_{k}(y).
\end{align}
From (\ref{2dph2}), we take $\varphi(x)=L_m(x)$, $m=0,1,\ldots,k-1$, to obtain
\begin{align}
\alpha_{m,k}=0, \quad m=0,\ldots,k-1.
\end{align}
By the same arguments, we have
\begin{align}
\alpha_{k,m}=0, \quad m=0,\ldots,k-1.
\end{align}
Thus $\Pi_h^\star u=\alpha_{k,k}L_{k}(x)L_{k}(y)$.  Finally, by (\ref{2dph4}) and the fact that $k$ is even, we have $\alpha_{k,k}=0$.
That means $\Pi_h^\star u\equiv 0$. We have now finished the proof of uniqueness, hence also existence. By the same arguments as
the proof of Lemma 2.1 in \cite{liu2019pk}, we can obtain (\ref{2dprojectionbound}).
\end{proof}
\subsection{Proof of Proposition \ref{2dprosuper}}
\label{proofofprop2d}
\begin{proof}
If $u\in \mathbb{Q}^{k}$, then $\Pi_h^\star u = u$ implies (\ref{2dprosuperequ}) holds true. Thus we only need to prove
the cases $u=x^{k+1}$ or $y^{k+1}$. We will just show the details of the proof for one case; namely $b_{i,j}(\Pi_h^\star u,v)=b_{i,j}(u, v), \ {\rev{\forall v \in \mathbb{Q}^k(K_{i,j})}}$, is true when $u=x^{k+1}$. We denote $\Pi_e=\Pi_h^\star u -u$. By the definition of $b_{i,j}(\cdot,\cdot)$, we have
\begin{align}
&b_{i,j}(\Pi_e,v)=\int_{K_{i,j}} \Pi_e v_x +\Pi_e v_y \, dx dy \nonumber\\
&-\int_{J_{j}}\frac{\Pi_e(x_{\ir}^+,y)+\Pi_e(x_{\ir}^-,y)}{2}v(x_{\ir}^-,y)-\frac{\Pi_e(x_{\il}^+,y)+\Pi_e(x_{\il}^-,y)}{2}v(x_{\il}^+,y) \, dy\nonumber \\
&-\int_{I_{i}}\frac{\Pi_e(x,y_{\jr}^+)+\Pi_e(x,y_{\jr}^-)}{2}v(x,y_{\jr}^-)-\frac{\Pi_e(x,y_{\jl}^+)+\Pi_e(x,y_{\jl}^-)}{2}v(x,y_{\jl}^+) \, dx.
\end{align}
We first have $(\Pi_h^\star u-u)_y=0$ due to the special form of $u$. {\rev{Since $v_x$ is a polynomial of degree at most $k-1$ in $x$}}, thus from (\ref{2dph1}), we have
\begin{align}
\int_{K_{i,j}} \Pi_e v_x \, dx dy =0\label{term1},
\end{align}
and since $\Pi_e$ is continuous corresponding to the variable $y$, after applying integration by parts, we obtain
\begin{align}
&\int_{K_{i,j}}\Pi_e v_y \, dx dy \nonumber\\ &-\int_{I_{i}}\frac{\Pi_e(x,y_{\jr}^+)+\Pi_e(x,y_{\jr}^-)}{2}v(x,y_{\jr}^-)-\frac{\Pi_e(x,y_{\jl}^+)+\Pi_e(x,y_{\jl}^-)}{2}v(x,y_{\jl}^+) \, dx\nonumber\\
&=-\int_{K_{i,j}}(\Pi_e)_y v\, dx dy=0.\label{term2}
\end{align}
From Lemma \ref{projectionlem}, we have
\begin{align}
\Pi_e(x_{\ir}^+,y)+\Pi_e(x_{\ir}^-,y)=&u(x_{\ir},y)-u(x_{\il},y)+\Pi_h^\star u(x_{\il}^+)-u(x_{\ir},y) \nonumber \\
&+\Pi_h^\star u (x_{\ir}^-,y)-u(x_{\ir},y)\nonumber \\
=&\Pi_h^\star u(x_{\il}^+)+\Pi_h^\star u (x_{\ir}^-,y)-u(x_{\il},y)-u(x_{\ir},y)\nonumber \\
=&0.\label{term3}
\end{align}
The last equality is from (\ref{2dph4}). By the same arguments,
\begin{align}\label{term4}
\Pi_e(x_{\il}^+,y)+\Pi_e(x_{\il}^-,y)=0.
\end{align}
From (\ref{term1})-(\ref{term4}), we have $b_{i,j}(\Pi_e,v)=0$.
\end{proof}

\end{document}